\documentclass{amsart}

\usepackage{amssymb, amsmath,textcomp}
\usepackage{hyperref}
\usepackage{srcltx}
\usepackage{color}

\renewcommand{\marginpar}[1]{}

\numberwithin{equation}{section}

\theoremstyle{plain}
\newtheorem{proposition}{Proposition}[section]
\newtheorem{theorem}[proposition]{Theorem}
\newtheorem{lemma}[proposition]{Lemma}

\newtheorem{corollary}[proposition]{Corollary}

\theoremstyle{definition}

\newtheorem{definition}[proposition]{Definition}

\newcommand{\Q}{\mathbb{Q}}
\newcommand{\C}{\mathbb{C}}
\newcommand{\N}{\mathbb{N}}
\newcommand{\Z}{\mathbb{Z}}
\newcommand{\R}{\mathbb{R}}

\newcommand{\defh}[1]{\langle #1 \rangle}

\newcommand{\bs}{\boldsymbol}

\def \<{\langle}
\def \>{\rangle}

\begin{document}

\bibliographystyle{plain}

\title{Rational solutions of polynomial-exponential equations}
\date{\today}

\author{Ayhan G\"unaydin}
\address{Centro de Matem\'atica e Aplica\c{c}oes Fundamentais (CMAF), Avenida Professor Gama Pinto, 2, 1649-003 Lisbon, Portugal }
\email{ayhan@ptmat.fc.ul.pt}

%
%

\maketitle

\section{Introduction}
Let ${\bs X}=(X_1,\dots,X_t)$ be a tuple of indeterminates. Fix $P_1,\dots,P_s\in\C[{\bs X}]$ and $\boldsymbol\alpha_{1},\dots,\boldsymbol\alpha_{s}\in\C^t$.  We consider the following equation
\begin{equation*}\label{mainEQ}\tag{\textborn}
 \sum_{i=1}^s P_i(\bs X)\exp(\bs X\cdot\bs\alpha_{i})=0,
\end{equation*}
where $\exp$ denotes the usual complex exponentiation map and $\bs X\cdot\bs\alpha_{i}$ is short for $X_1\alpha_{i1}+\cdots+X_t\alpha_{it}$. 

\medskip\noindent
A solution $\bs a=(a_1,\dots,a_t)\in\C^t$ of (\ref{mainEQ}) is said to be {\it nondegenerate} if 
$$\sum_{i\in I}P_i(\bs a)\exp(\bs a\cdot\bs\alpha_{i})\neq 0$$ 
for every nonempty proper subset $I$ of $\{1,\dots,s\}$.  

\medskip\noindent
Consider the $\Q$-vector space 
$$V:=\{\bs q\in\Q^t:\bs q\cdot\bs\alpha_i=\bs q\cdot\bs\alpha_{i'}\text{ for all } i,i'\}.$$
Fix a complement $V'$ of $V$ in $\Q^t$ and let $\pi:\Q^t\to V$ and $\pi':\Q^t\to V'$ denote the natural projections.

\medskip\noindent
We give a description of nondegenerate rational solutions of (\ref{mainEQ}) as follows.

\begin{theorem}\label{maintheorem}
Given $P_1,\dots,P_s\in\C[\bs X]$ and $\bs\alpha_{1},\dots,\bs\alpha_{s}\in\C^t$, there is $N\in\N^{>0}$ 
such that if $\bs q\in\Q^{t}$ is a nondegenerate solution of (\ref{mainEQ}), then $\pi'(\bs q)\in (\frac{1}{N}\Z)^t$.
\end{theorem}

\medskip\noindent
Note that searching for $\bs n=(n_1,\dots,n_t)\in\Z^t$ satisfying (\ref{mainEQ}) amounts to solving the classical polynomial-exponential equation
  
\begin{equation*}\label{mainEQ2}\tag{\textborn\textborn}
 \sum_{i=1}^{s}P_i(\bs n)\prod_{j=1}^{t}\beta_{ij}^{n_j}=0,
\end{equation*}
where $\beta_{ij}=\exp(\alpha_{ij})$ $(i=1,\dots,s, j=1,\dots,t)$.
This kind of solutions were considered in \cite{LaurentII}.  For such an equation let 
$$H=\{\bs n\in\Z^t:\prod_{j=1}^{t}\beta_{ij}^{n_j}=\prod_{j=1}^{t}\beta_{i'j}^{n_j}\text{ for every }i,i'\}.$$

\begin{theorem}\label{Laurent} (Laurent~\cite{LaurentII}) (i) If $P_i$ is constant for each $i$, then the set of nondegenerate solutions of (\ref{mainEQ2}) is a finite union of translates of $H$.

\smallskip\noindent
(ii) There are constants $a,b\in \R$ depending on the $P_i$'s and the $\beta_{ij}$'s such that if $\bs n$ is a nondegenerate solution of (\ref{mainEQ2}), then there is $\bs n'\in H$ with 
$$|\bs n-\bs n'|< a \log(|\bs n|)+b.$$
 
\end{theorem}

\noindent({\it Here and below $|(b_1,\dots,b_l)|:=\max\{|b_1|,\dots,|b_l|\}$.})

\medskip\noindent
This is a slightly weaker version of Th\'{e}or\`{e}me 1 of ~\cite{LaurentII}, because there the author considers not only nondegenerate solutions, but solutions that are `maximal' with respect to a partition of $\{1,\dots,s\}$.

%

\medskip\noindent
Combining Theorem \ref{maintheorem} with Laurent's result, we get a finer description of rational solutions of (\ref{mainEQ}).

\begin{corollary}\label{maincorollary1}
Given $P_1,\dots,P_s\in\C[\bs X]$ and $\bs\alpha_{1},\dots,\bs\alpha_{s}\in\C^t$, there are $N\in\N^{>0}$ and $c_{\bs r},d_{\bs r}\in\R$ for each $\bs r\in V$ 
such that if $\bs q\in\Q^{t}$ is a nondegenerate solution of (\ref{mainEQ}), then $\pi'(\bs q)\in (\frac{1}{N}\Z)^t$ and there is $\bs m\in\Z^t$ satisfying  $\bs m\cdot(\bs\alpha_i-\bs\alpha_{i'})\in\Z 2N\pi\sqrt{-1}$ for each $i,i'$ such that
 $$|\pi'(\bs q)-\frac{\bs m}{N}|<c_{\pi(\bs q)}\log(|\pi'(\bs q)|)+d_{\pi(\bs q)}.$$
\end{corollary}


\medskip\noindent
We also have the following finiteness result as a consequence of Corollary \ref{maincorollary1}.

\begin{corollary}\label{maincorollary}
 Let $P_1,\dots,P_s\in\C[X]$ and $\bs\alpha_{1},\dots,\bs\alpha_{s}\in\C^t$. Suppose that the set $\{\beta_{ij}:i=1,\dots,s, j=1,\dots,t\}$ is multiplicatively independent. Then there are only finitely many nondegenerate solutions $\bs q\in\Q^t$ of (\ref{mainEQ}).
\end{corollary}

\section{Linear relations in multiplicative groups}

\noindent
In this section we recall some notations and an earlier result that will be useful in the rest of the paper and we make the first reduction.

\medskip\noindent
Let $K$ be any field and $\Gamma$ a subgroup of $K^\times$.
We consider 
solutions in $\Gamma$ of 
\begin{equation*}\label{linearEQ}\tag{*}
 \lambda_1 x_1+\cdots+\lambda_k x_k=1,
\end{equation*}
where $\lambda_1,\ldots,\lambda_k\in K$. We say that a solution $\bs\gamma=(\gamma_1,\dots,\gamma_k)$ in $\Gamma$ 
of (\ref{linearEQ}) is {\it non-degenerate} if $\sum\limits_{i\in I} \lambda_i \gamma_i\neq 0$ for every nonempty proper 
subset $I$ of $\{1,\ldots,k\}$. 


\begin{definition}
Let $G$ be an abelian group, written multiplicatively. A
subgroup $H$ of $G$
is \emph{radical} (in 
$G$) if for each $n>0$ and $g\in G$ we have $g\in H$ whenever $g^n\in H$.

\medskip\noindent
Given $A\subseteq G$, we set $\defh{A}_G$ to be the smallest radical subgroup of $G$ containing $A$. That is,
$$ \defh{A}_G =\{ g\in G\, |\,g^n \in [A]_G \text{ for some } n\in\N\}$$
where $[A]_G$ denotes the subgroup generated by $A$.  When $G$ is clear from the context, we will drop the subscripts 
and just write $\defh{A}$ and $[A]$.
\end{definition}

\medskip\noindent
Also in what follows, $\mathbb U$ denotes the multiplicative group of roots of unity.

\medskip\noindent  We use the following result from \cite{MP}.

\begin{lemma}\label{extension}
Let $E\subseteq F$ be fields such that $E\cap\mathbb U=F\cap\mathbb U$ and $G$ be a radical subgroup of $E^\times$.
Then for $\lambda_1,\dots,\lambda_n\in E^\times$, the equation (\ref{linearEQ}) 
has the same non-degenerate solutions in $G$ as in $\defh{G}_{F^\times}$.
\end{lemma}


\medskip\noindent
In the rest of the paper, it will be more convenient to consider the following equation rather than (\ref{mainEQ})
\begin{equation}\label{mainEQ4}\tag{\textborn\textborn\textborn}
 \sum_{i=1}^s P_i(X)\exp(X\cdot\bs\alpha'_{i})=0,
\end{equation}
where $\bs\alpha'_{i}=(\alpha_{ij}-\alpha_{1j})_{j=1,\dots,t}$.
Note that $\bs q$ is a nondegenerate solution of (\ref{mainEQ}) if and only if it is a nondegenerate solution of (\ref{mainEQ4}). Hence we do not loose any information by replacing $\bs\alpha_i$ with $\bs\alpha'_i$. Let $S$ denote the set of nondegenerate solutions $\bs q\in\Q^t$ of (\ref{mainEQ4}), and put $\beta'_{ij}=\exp(\alpha_{ij}-\alpha_{1j})$.
 
\medskip\noindent
Let $A$ be a finite set containing the coefficients of the $P_i$'s and the $\beta'_{ij}$'s and put $\Gamma=\<A\>_{\C^\times}$. If $\bs q\in S$, then $P_1(\bs q)\neq 0$ and the tuple
$(\exp(\bs q\cdot\bs\alpha'_{i}):i=2,\dots,s)\in\Gamma^{s-1}$ is a non-degenerate solution of the linear equation
\begin{equation*}\label{linearEQ2}\tag{**}
 \frac{P_2(\bs q)}{-P_1(\bs q)}Y_2+\cdots+ \frac{P_s(\bs q)}{-P_1(\bs q)}Y_s=1.
\end{equation*}
Note that when $\bs q$ varies in $\Q^t$ with $P_1(\bs q)\neq 0$, the coefficients of this equation vary in the field $\Q(A)$.

\medskip\noindent
Let $E:=\Q(\mathbb U\cup A)$ and $G:=\<A\>_{E^\times}$. Now by taking $\C$ in the place of $F$ in Lemma \ref{extension}, we see that all the possible solutions of the linear equation (\ref{linearEQ2}) in $\Gamma$ are in $G$.  

\medskip\noindent
Let $G'$ be the complement of $\mathbb U$ in $G$. 
Therefore if $\bs q\in S$, then there are roots of unity $\zeta_{\bs q 2},\dots,\zeta_{\bs q s}$ and $\eta_{\bs q 2},\dots,\eta_{\bs q s}\in G'$ such that
\begin{equation*}\label{roots_of_1}\tag{***}
 P_1(\bs q)+P_2(\bs q)\eta_{\bs q 2}\,\zeta_{\bs q 2}+\dots+P_s(\bs q)\eta_{\bs q s}\,\zeta_{\bs q s} =0.
\end{equation*}
We elaborate on the relation between $\bs q$ and the $\eta_{\bs q i}$'s later; we first conclude that there is a bound on the order of the $\zeta_{\bs q i}$'s.


\section{Roots of unity}
We first remark the following easy observation, whose proof follows the lines of the proof of Lemme 4 of \cite{LaurentII}
(Note that Laurent considered only finitely generated 
$\Q$-algebras, however his result is deeper).

\begin{lemma}\label{LinInd_spe}
Let $R$ be a subring of $\bar\Q[S]$, where $S$ is a finite subset of $\C$. Suppose that $b_1,\dots,b_q$ 
are elements of $R$ and let $q'$ be the linear dimension over $\bar\Q$ of $\bs b$.  
Then there are ring homomorphisms $\phi_1,\dots,\phi_{q'}$ from $R$ to $\bar\Q$ fixing $k:=R\cap\bar\Q$ such 
that for every $\alpha_1,\dots,\alpha_{q}\in k$ with $\alpha_1 b_1+\cdots+\alpha_q b_q\neq 0$ there is
 some $i\in\{1,\dots,q'\}$  with $\phi_i(\alpha_1 b_1+\cdots+\alpha_q b_q)\neq 0$.
\end{lemma}

\begin{proof}
 After changing $R$, we may assume that $q=q'$. It suffices then to find  $\phi_i:R\to \bar\Q$ for each 
 $i\in\{1,\dots,q\}$ fixing $k$ such that the determinant
$$D_q:=\left| \begin{array}{ccc} \phi_1(b_1) & \cdots & \phi_1(b_q) \\ \vdots & \  & \vdots \\ \phi_q(b_1) & \cdots 
& \phi_q(b_q)\end{array} \right| $$
is nonzero.

\medskip\noindent
We proceed by induction on $q$. For $q=1$ by Nullstellensatz take a ring homomorphism $R[b_1^{-1}]\to \bar\Q$ that fixes 
$k$. Clearly, its restriction to $R$ sends $b_1$ to some non-zero element. 

\medskip\noindent
Assume now that  $\phi_1,\dots,\phi_{q-1}$ have been already constructed such that $D_{q-1}\neq 0$. Then the determinant
$$D'_q:=\left| \begin{array}{ccc} \phi_1(b_1) & \cdots & \phi_1(b_q) \\ \vdots & \  & \vdots \\ \phi_{q-1}(b_1) & \cdots 
& \phi_{q-1}(b_q) \\  b_1 & \cdots & b_q\end{array} \right| $$
is $\beta_1 b_1+\cdots+\beta_q b_q$, where $\beta_1,\dots,\beta_q$ are algebraic numbers. In particular, by induction,
 $\beta_q=D_{q-1}\neq 0$. Therefore, since we are assuming that the tuple $\bs b$ is $\bar \Q$-linearly independent, we conclude 
that $D'_q\neq 0$. Consider $$R':=R[ (D'_q)^{-1}].$$
Nullstellensatz implies that there is a ring morphism $\phi_q$ from $R'$ to $\bar\Q$ fixing $k':=R'\cap\bar\Q$. 
Its restriction to $R$ has the property that $\phi_q D'_q\neq 0$ which implies that $D_q\neq 0$.
\end{proof}

\medskip\noindent
In order to bound the degrees of the roots of unity appearing in (\ref{roots_of_1}) we need  
the following result.

\begin{theorem}\label{dvornicich-zannier}{(Theorem 1 in \cite{dvornicich-Zannier}\!)}\\
 Let $F$ be a number field, $a_0,a_1,\dots,a_k$ in $F$ and $\zeta$ a root of unity of order $Q$ such that 
$a_0+\sum\limits_{j=1}^{k}a_j\zeta^{n_j}=0$ with $\gcd(Q,n_1,\dots,n_k)=1$.  Let $\delta=[F\cap\Q(\zeta):\Q]$ 
and suppose that for any nonempty proper subset $I$ of $\{0,1,\dots,k\}$ the sum 
$\sum\limits_{j\in I}a_j\zeta^{n_j}\neq 0$. Then for each prime $p$ and $n>0$, if $p^{n+1}|Q$, 
then $p^{n}|2\delta$ and
 $$k\geq \dim_F(F+F\zeta^{n_1}+\cdots+F\zeta^{n_k})\geq 1+\sum_{p|Q,p^2\nmid  Q}[\frac{p-1}{\gcd(\delta,p-1)}-1].$$
In particular, the order $Q$ of $\zeta$ is bounded by a constant depending on $k$ and $\delta$ (and therefore $[F:\Q]$).
\end{theorem}

\medskip\noindent
We are ready to prove the following.

\begin{lemma}\label{bound_roots_of_1}
There is $T\in\N$ depending only on the coefficients of the $P_i$'s and the $\beta_{ij}$'s such that if $\bs q$ is a nondegenerate solution of (\ref{mainEQ4}), then all the corresponding $\zeta_{\bs q i}$'s are of order less than $T$.
\end{lemma}

\begin{proof}
Let 
$$\zeta_{\bs q i}=\exp(\frac{l_{\bs q i}2\pi\sqrt{-1}}{n_{\bs q i}}),$$
with $l_{\bs q i}< n_{\bs q i}$ and $\gcd(l_{\bs q i},n_{\bs q i})$.
We need to find a bound $T$ on the $n_{\bs i}$'s independent of $\bs q$. 


\medskip\noindent
Let $R$ be the $\bar \Q$-algebra generated by elements of $A$ and their inverses. Using Lemma \ref{LinInd_spe} with appropriate $b_1,\dots,b_q$, choose by Lemma \ref{LinInd_spe} some specialization $\phi$ such that 
$$\phi(P_1(\bs q))\neq 0.$$

\medskip\noindent
The homomorphism $\phi$ transforms (\ref{roots_of_1}) into a  non-trivial relation 
$$\sum_{i=1}^s \phi(P_i(\bs q))\phi(\eta_{\bs q i})\,\zeta_{\bs q i}=0.$$

\medskip\noindent
  So we have a relation
$$\sum_{j\in I}a_j \zeta^j=0,$$
where $\zeta=\exp(\frac{2\pi\sqrt{-1}}{N})$ with $N=\operatorname{lcm}(n_{\bs q i}:i=1,\dots,s)$ and the $a_j$'s are algebraic numbers depending on $\bs q$ and not all zero.

\medskip\noindent
For our purposes we may assume that no subsum is $0$. Then Theorem \ref{dvornicich-zannier} gives that $T$
depends only on the degree of $F$ and $|I|$ and not on $\bs q$.
\end{proof}

\section{Proof of Theorem \ref{maintheorem}}
\medskip\noindent
In order to finish the proof of Theorem \ref{maintheorem}, we need the following Kummer theoretic result from \cite{Bays-Zilber}.
\begin{proposition}\label{P:BaysZil}
{\it (Bays-Zilber)}  Let $L$ be a finitely generated extension of $\Q(\mathbb U)$. Then the quotient group $L^\times/\mathbb U$ is a free abelian group.
\end{proposition}

\medskip\noindent
It follows from this proposition that the group $G'$ from Section 2 is finitely generated. Indeed $G'$ can be considered as a subgroup of $E^\times/\mathbb U$, which is a free abelian group by the proposition and being of finite rank $G'$ is actually finitely generated.  Then by using Lemma \ref{bound_roots_of_1}, if  $\bs q\in S$, then $(\bs q\cdot\bs\alpha'_2,\dots,\bs q\cdot\bs\alpha'_s)$ is in a finitely generated subgroup $A$ of 
$$W:=\{(\bs q\cdot\bs\alpha'_2,\dots,\bs q\cdot\bs\alpha'_s):\bs q\in\Q^t\}.$$
Note that $W$ is a $\Q$-linear subspace of $\C^{(s-1)}$ and there is a natural surjective linear map $\phi:\Q^t\to W$ taking $\bs q$ to $(\bs q\cdot\bs\alpha'_2,\dots,\bs q\cdot\bs\alpha'_s)$.  Then the kernel of $\phi$ is $V$ from the Introduction. 

\medskip\noindent
The preimage of $A$ under $\phi$ is of the form $V\oplus B$ where $B$ is a finitely generated subgroup of $V'$.  Note that the rank of $B$ is at most the dimension of $V'$.

\medskip\noindent
Let $B=\Z \bs r_1\oplus\dots\oplus\Z\bs r_d$, where $\bs r_1,\dots,\bs r_d\in\Q^t$.  For $k=1,\dots,d$, write $\bs r_k$ as $(\frac{a_{k1}}{b_{k1}},\dots,\frac{a_{kt}}{b_{kt}})$, where $a_{kj}\in\Z$, $b_{kj}\in\N^{>0}$ with $\gcd(a_{jk},b_{jk})=1$ for each $j=1,\dots,t$. 

\medskip\noindent
Now let 
$$N:=\operatorname{lcm}(b_{kj}:j=1,\dots,t,k=1,\dots,d).$$ 
Then $B\subseteq (\frac{1}{N}\Z)^t$. Hence if $\bs q\in S$, then $\pi'(\bs q)\in (\frac{1}{N}\Z)^t$, finishing the proof of Theorem \ref{maintheorem}.

\section{Final remarks}
\noindent
The motivation for this work is to answer the following question affirmatively.

\medskip\noindent
{\bf Question}. Assume {\it Schanuel's conjecture}. Is it correct that for every irreducible polynomial $p(X,Y)\in\C[X,Y]$ in which both $X$ and $Y$ appear, there is a generic point of the form $(\alpha,\exp(\alpha))$ on the complex curve defined by $p$?

\medskip\noindent
Schanuel's conjecture is the statement that for every $\Q$-linearly independent complex numbers $\alpha_1,\dots,\alpha_n$, we have 
$$\operatorname{trdeg}_{\Q} \Q(\alpha_1,\dots,\alpha_n, \exp(\alpha_1),\dots,\exp(\alpha_n))\geq n.$$

\medskip\noindent
The question above stems from the model theoretic study of the field of complex numbers expanded by the usual exponential function, $(\C,\exp)$. In \cite{zilber_exp}, Zilber gives an axiomatization of a first order theory and conjectures that $(\C,\exp)$ is a model of that theory.  Since one of the axioms is Schanuel's conjecture, it seems like Zilber's conjecture is out of reach. However, one can try to reduce it to Schanuel's conjecture.  An affirmative answer to the question above would be a first step for such a purpose.

\medskip\noindent
To find such a generic point, it is sufficient to show that for each subfield $K$ of $\C$ of finite transcendence degree, there are only finitely many $\alpha\in K$ such that $p(\alpha,\exp(\alpha))=0$ (see \cite{marker_exp} for an explanation of this).  So let $K$ be such a subfield. Using Schanuel's conjecture one can conclude that $\alpha\in K$ with $p(\alpha,\exp(\alpha))=0$ lie in a $\Q$-linear subspace of $K$ of finite dimension. Then the question reduces to understanding the rational solutions of certain polynomial-exponential equations.  Unfortunately, the description we have in this article does not imply that there are only finitely many such solutions.  However, it is still possible that restricting to the kind of equations that occur as above we have a better description of rational solutions.

\bibliography{ref}

\end{document}